\documentclass[10pt]{amsart}
\usepackage{graphicx}
\usepackage{amsmath}
\usepackage{amssymb}
\usepackage{xypic}
\usepackage{enumerate}
\newtheorem{theorem}{Theorem}%[section]
\newtheorem{lemma}[theorem]{Lemma}
\newtheorem{corollary}[theorem]{Corollary}
\newtheorem{proposition}[theorem]{Proposition}
\theoremstyle{remark}
\newtheorem{remark}[theorem]{Remark}
\theoremstyle{remark}
\newtheorem{example}[theorem]{Example}

\newcommand\Rmd{{\text{\rm Mod-}}}

\newcommand{\Hom}{\mathop{\rm Hom}\nolimits}

\newcommand{\End}{\mathop{\rm End}\nolimits}

\newcommand{\Img}{\mathop{\rm Im}\nolimits}

\newcommand\CS{{\mathcal S}}

\newcommand\N{{\mathbb N}} % the set of natural numbers
\newcommand\Z{{\mathbb Z}} % the set of integers
\newcommand\J{{\mathbb J}} % the set of ?-adic integers
\newcommand\Q{{\mathbb Q}} % the set of rationals
\newcommand\id{{\bf 1}}

\newcommand{\m}{\mathfrak{m}}

\begin{document}

\title{Constructing large self--small modules}
\author{George Ciprian Modoi}
\address{Babe\c s--Bolyai University, Faculty of Mathematics and Computer Science \\  1, Mihail Kog\u alniceanu, 400084 Cluj--Napoca, Romania}
\email{cmodoi@math.ubbcluj.ro}

%\thanks{}


\date{\today}

\begin{abstract} We give a method for constructing (possible large) self--small modules via some special homomorphisms of rings, called here weak epimorphisms. 
\end{abstract}
\maketitle

Various kind of smallness appear naturally in the  study of situations in which the covariant or contravariant hom functor induces an equivalence, 
respectively a duality between some categories of modules. For example Morita theory says that if $R$ is an arbitrary ring, $P$ is a progenerator in the category $\Rmd R$ of right 
$R$-modules and $E=\End_R(P)$ is its endomorphism ring, then 
the functor $\Hom_R(P,-):\Rmd R\to\Rmd E$   is an equivalence, with the inverse the tensor product $-\otimes_EP$. In these conditions, $P$ has to be {\em small}, 
that is $\Hom_R(P,-)$ has to commute with arbitrary direct sums. 

The smallness notion can be generalized in various ways, by imposing some restrictions to the class of direct sums which the covariant hom functor has to commute. 
In this note we deal with the following generalization:  A {\em self--small} $R$-module
is a module $M$ such that $\Hom_R(M,M^{(I)})\cong\Hom(M,M)^{(I)}$, naturally for every set $I$.
Self--small abelian groups (that is, $\Z$-modules) were introduced by Arnold an Murley in \cite{AM75}. 
The relevance of the study of self--small abelian groups is justified by many papers (see, for example, \cite{ABW09} and the references therein).

In this note we want to construct a module which is self--small but it is large in some sense. More precisely, we want this self--small module to be not small. 
Because finitely generated modules are always small, the modules we are looking for have to be infinitely generated. The 
method is inspired by the construction of the abelian group of $p$-adic integers $\J_p$, where $p$ is a prime. In this case, $\J_p$ is uncountable, that is 
its cardinality is also larger than the cardinality of the ring of integers $\Z$.   

Note that another way of constructing large self--small modules can be found in \cite{Z08}. More precisely, from \cite[Example 2.7]{Z08} we learn that the direct product $\prod_{p}\Z/p$ 
is self--small, but the direct sum $\bigoplus_{p}\Z/p$ is not, where $p$ runs over all primes and 
$Z/n=\Z/n\Z$ for every $n\in\N$.  
More generally, for a ring $R$ let denote by $\CS_R$ a representative set of simple modules. Then in \cite[Theorem  2.5 and Corollary 1.3]{Z08} we find some sufficient conditions for the 
direct product $\prod_{S\in\CS_R}S$ to be, respectively to be not self--small.

In what follows  we consider two rings with one $R$ and $J$, and we denote by $\Rmd R$ and $\Rmd J$ the respective categories of modules (which by default are left modules). 
Let $\varphi:R\to J$ a unitary ring
homomorphism. Thus $J$ has a natural structure of $R-R$-bimodule
and $\varphi$ induces a pair of adjoint functors (the restriction
and the induction of the scalars):
\[\varphi_*=\Hom_J(J,-):\Rmd J\rightleftarrows
\Rmd R:(-\otimes_RJ)=\varphi^*.\] The restriction functor $\varphi_*$ acts as follows: $\varphi_*(M)=M$ and $ax=\varphi(a)x$ for all $J$-modules $M$ and all $x\in M$ and $a\in R$. Henceforth it is obviously faithful, since it sends a $J$-linear map in itself, but seen as $R$-linear.

Recall that $\varphi$ is called an {\em epimorphism of rings}, if for every two parallel homomorphisms of rings $\psi,\zeta:J\to J'$ we have  
\[\psi\cdot\varphi=\zeta\cdot\varphi\Rightarrow\psi=\zeta.\] This 
is the case, exactly if $\varphi_*$ is full too, by \cite[Ch. XI, Proposition 1.2]{S}, therefore if we have $\Hom_R(M,N)\cong\Hom_J(M,N)$ for all $M,N\in\Rmd J$. Inspired by this, we call $\varphi$ {\em weak epimorphism} if $\Hom_R(J,J)\cong\Hom_J(J,J)$, that is $\Hom_R(J,J)\cong J$.

\begin{proposition}\label{qepiss}
 If $\varphi: R\to J$ is a weak epimorphism of rings, then $J$ is self--small as $R$-module. 
\end{proposition}

\begin{proof}
 Let $I$ be a set and denote by $\pi_i:J^{(I)}\to J$ the projection of the coproduct of copies of $J$ into its $i$-th component ($i\in I$). 
 If $f:J\to J^{(I)}$ is an arbitrary $R$-linear map, then $\pi_if:J\to J$ is $R$-linear for all $i\in I$.  According to our hypothesis it is $J$-linear too, therefore it is determined by $\pi_if(1)\in J$. Because $\pi_if(1)\neq 0$ only for a finite number of $i$'s, we conclude that $\pi_If=0$ for almost all $i\in I$, hence $f$ factors through a finite subcoprodct of $J^{(I)}$, what is the same as saying that $\Hom_R\left(J,J^{(I)}\right)\cong\Hom(J,J)^{(I)}$.   
\end{proof}

Since epimorphisms of rings are obviously weak epimorphisms too we obtain:  

\begin{corollary}\label{episs}  If $\varphi:R\to J$ is an epimorphism of rings, then the $R$-module $J$ is self--small.
\end{corollary}

\begin{example}
 The inclusion $\Z\to\Q$ is known to be an epimorphism of rings, namely one which is not surjective. Therefore 
 Corollary \ref{episs} above gives us a new proof that the abelian group  $\Q$ is self--small.
\end{example}

In the sequel we assume that the ring $R$ is commutative.
Thus $\Rmd R$ coincide to the category of right $R$-modules, and $\Hom_R(M,N)$ is an $R$-module for all $M,N\in\Rmd R$. In
$\Rmd R$ consider an ascending chain of submodules
\[{\rm (DS)}\ Z_1\overset{\mu_1}\to Z_2\overset{\mu_2}\to Z_3\to\ldots,\]
of the module
\[Z_\infty=\displaystyle{\lim_\to}Z_n=\lim_\to(Z_1\overset{\mu_1}\to
Z_2\overset{\mu_2}\to Z_3\to\ldots)=\bigcup Z_n,\] the morphisms
$\mu_n$ being inclusions. 
Relative to the above chain consider the following conditions:
\begin{enumerate}[(1)]
\item All modules $Z_m$ are finitely presented. \item
$\Hom_R(Z_m,Z_n)\cong Z_m$ naturally, for all $1\leq m\leq n$.
\item The $R$-module $Z_\infty$ is injective relative to all exact
sequences \[0\to Z_m\overset{\mu_n}\to Z_{m+1}\to
Z_{m+1}/Z_m\to0,\] with $m\geq 1$. 
\item $Z_1$ is simple, and denote by $U$ the annihilator of $Z_1$ in $R$, that is $U$ is a maximal ideal in $R$ and there is a short exact sequence \[0\to U\to R\to Z_1\to 0.\]
Moreover assume that $Z_{m+1}U=Z_n$, for all $m\in\N^*$.
\item $Z_m\otimes Z_1\cong Z_1$ naturally, for all $m\in\N^*$. 
\end{enumerate}

Note that the condition (3) is automatically satisfied, if we know that the 
$R$-module $Z_\infty$ is injective. 
On the other hand we can replace (3) with a condition relative to the direct
system (DS), rather than relative to its direct limit $Z_\infty$, as in the the
following:

\begin{itemize}
\item[(3')] The $R$-module $Z_n$ is injective relative to all
exact sequences \[0\to Z_m\overset{\mu_n}\to Z_{m+1}\to
Z_{m+1}/Z_m\to0,\] with $1\leq m<n$.
\end{itemize}

\begin{lemma}\label{znforz} If (1) and (3') are satisfied then
(3) holds too.
\end{lemma}

\begin{proof} The condition (3') implies that the induced
homomorphism
\[\Hom_R(Z_{m+1},Z_n)\to\Hom_R(Z_m,Z_n)\] is
surjective for all $1\leq m<n$. The condition (1) says that all
$Z_{i}$, $i\geq1$ are
finitely generated, and this means the functors
$\Hom_R(Z_i,-)$ commute with direct limits as we can 
see from \cite[Ch. V, Proposition 3.4]{S}. We deduce that the induced homomorphism
\begin{align*} \Hom_R(Z_{m+1},Z_\infty)&\cong\lim_\to\Hom_R(Z_{m+1},Z_n)\\
&\to\lim_\to\Hom_R(Z_m,Z_n)\cong\Hom_R(Z_m,Z_\infty) \end{align*} is also
surjective, therefore (3) holds.
\end{proof}

\begin{lemma}\label{homzz} If (1) and (2) hold, we have for all $m\geq 1$ a natural
isomorphism:
 \[\Hom_R(Z_m,Z_\infty)\cong Z_m.\]
\end{lemma}

\begin{proof} Using again the property that
$\Hom_R(Z_i,-)$ commutes with direct limits, for all $i\geq 1$, we get:
\[\Hom_R(Z_m,Z_\infty)=\Hom_R(Z_m,\lim_\to
Z_n)\cong\lim_\to\Hom_R(Z_m,Z_n)\cong Z_m,\] where the last isomorphisms follows from the fact that 
(2) implies that the direct system $\{\Hom_R(\Z_m,\Z_n)\}_{n\geq1}$ looks like
\[\Hom_R(Z_1,Z_m)\to\ldots \to\Hom_R(Z_{m-1},Z_m)\to Z_m\overset{=}\to Z_m\overset{=}\to Z_m\overset{=}\to\ldots,\] that is it has a cofinal constant subsystem. 
\end{proof}

Assume that (1) and (2) hold. For all $n\geq 1$, we denote  $\delta_n$ the composed homomorphism
\[Z_{n+1}\overset{\cong}\to\Hom_R(Z_{n+1},Z_\infty)\overset{(\mu_n)_*}\to\Hom_R(Z_n,Z_\infty)\overset{\cong}\to
Z_n,\] where the isomorphisms are coming from Lemma \ref{homzz}. We obtain
an inverse system of $R$-modules
\[{\rm (IS)}\ Z_1\overset{\delta_1}\leftarrow Z_2\overset{\delta_2}\leftarrow Z_3\leftarrow\ldots\]
Let now denote $J=\End_R(Z_\infty)$. Thus $J$ is naturally an
$R$-algebra, and let $\varphi:R\to J$ denote the structure
homomorphism of this algebra.

\begin{lemma}\label{jislim} If (1) and (2) hold, we have a natural isomorphism in $\Rmd R$:
\[J\cong\lim_\leftarrow Z_n=\lim_\leftarrow(Z_1\overset{\delta_1}\leftarrow Z_2\overset{\delta_2}\leftarrow
Z_3\leftarrow\ldots).\]
\end{lemma}

\begin{proof} The chain of isomorphisms (the last one coming from Lemma \ref{homzz})
\[J=\Hom_R(Z_\infty,Z_\infty)\cong\Hom_R(\lim_\to Z_n,Z_\infty)\cong\lim_\leftarrow\Hom_R(Z_n,Z_\infty)\cong\lim_\leftarrow Z_n. \]
proves our lemma.
\end{proof}

For the inverse system (IS) we denote $\delta_{jj}=1_{Z_j}$ and 
$\delta_{ji}=\delta_j\ldots\delta_i$, for all $1\leq j\leq i$.  With these notations, the  inverse system is called {\em Mittag--Leffler} if for each 
$k\geq 1$ there is $j>k$ such that  $\Img(\delta_{ki})=\Img(\delta_{kj})$ 
for all $j\leq i$. In particular this is always true, provided that the homomorphisms $\delta_i$ are surjective, for  all $i\geq 1$. 

\begin{lemma}\label{zml} If (1), (2) and (3) hold, then the inverse
system (IS) is
Mittag--Leffler.
\end{lemma}

\begin{proof} The homomorphism $(\mu_n)_*$ is surjective by (3), so the same
property is true for $\delta_n$, and the conclusion follows.
\end{proof}

From now on, we assume that all conditions (1)-(5) hold. 

\begin{lemma}\label{n+m/n}
We have $Z_{n+m}/Z_m\cong Z_n$ for all 
$n,m\in\N^*$.
\end{lemma}

\begin{proof} First we will show that $Z_{n+1}/Z_n\cong Z_1$ for all 
$n\in\N^*$. Indeed, applying the functor $Z_{n+1}\otimes_R-$ to the short exact sequence $0\to U\to R\to Z_1\to 0$, keeping in the mind that 
$Z_n=\Z_{n+1}U$ is the image of the map $Z_{n+1}\otimes_RU\to Z_{n+1}\otimes_RR$ and using condition (5) for the isomorphism in the last vertical arrow, we get a commutative diagram with exact rows 
\[\diagram 
&Z_{n+1}\otimes_RU\rto\dto&Z_{n+1}\otimes_RR\rto\dto^{\cong}&Z_{n+1}\otimes_RZ_1\rto\dto^{\cong}&0 \\
0\rto&Z_n\rto&Z_{n+1}\rto&Z_1\rto&0
\enddiagram\]
which proves our claim.  
 
 Fix $n\in\N^*$ and proceed by induction on $m$. 
 For $m=1$, we apply the functor $\Hom_R(-,Z_\infty)$ to the exact sequence from the second row of the last diagram. According to (3), we get an exact sequence too, which by Lemma \ref{homzz} looks like: 
 \[0\to Z_1\to Z_{n+1}\to Z_n\to 0,\] proving our desired isomorphism 
 $Z_{n+1}/Z_1\cong Z_n$.

Suppose now that  
 $Z_{n+m}/Z_n\cong Z_m$. Then construct the diagram having exact rows and columns (the exactness of the rows is shown in the first part of this proof, the induction hypothesis gives exactness of the first column, and for the second column it is obvious):
 \[\diagram
 &0\dto&0\dto&& \\
 0\rto&Z_m\dto\rto&Z_{m+1}\dto\rto&Z_1\ddouble\rto&0 \\
 0\rto&Z_{n+m}\dto\rto&Z_{n+m+1}\dto\rto&Z_1\rto&0 \\
 &Z_n\dto&Z_{n+m+1}/Z_{m+1}\dto&& \\
 &0&0&&
 \enddiagram\]
 Now the Ker-Coker lemma gives us the isomorphism $Z_{n+m+1}/Z_{m+1}\cong Z_n$.
\end{proof}

\begin{remark}
 Puttig together above lemmas, we deduce that for all $n,m\in\N^*$ we have the short exact sequences \[0\to Z_n\to Z_{n+m}\to Z_m\to 0\hbox{ and }0\to Z_n\to Z_{n+m}\to Z_m\to 0\] and the functor $\Hom_R(-,Z_\infty)$ sends them to each other. 
\end{remark}

\begin{lemma}\label{jjz}
 There is a short exact sequence 
 \[0\to J\stackrel{u}\to J\to Z_1\to 0\]
  such that $\Img u=UJ$. 
\end{lemma}

\begin{proof}
Consider the diagram with exact columns: 
 \[\diagram 
 0\dto&0\dto&0\dto& \\
 0\dto&Z_1\lto\dto&Z_2\lto\dto&\cdots\lto \\
 Z_1\dto&Z_{2}\lto\dto&Z_{3}\lto\dto&\cdots\lto \\
 Z_1\dto&Z_1\lto_{=}\dto&Z_1\lto_{=}\dto&\cdots\lto \\
 0&0&0& 
 \enddiagram\]
 Note that the involved inverse systems are Mittag Leffler by Lemma \ref{zml}, therefore the their inverse limits are exact by \cite[Theorem 5]{ML}. Therefore the inverse limit gives us the desired short exact sequence. 
 
 By its construction the homomorphism $u$ acts as follows: for all
 $(x_1,x_2,x_3,\ldots)\in J$ (that is 
 $(x_1,x_2,x_3,\ldots)\in\prod_{n\geq 1}Z_n$ such that $\delta_n(x_{n+1})=x_n$, for all $n$) we have $u(x_1,x_2,x_3,\ldots)=(0,x_1,x_2,\ldots)$, so $UZ_{n+1}=Z_n$, for all $n\geq 1$ implies $UJ=\Img u$. 
\end{proof}

\begin{lemma}\label{homjz}
 The following sentences hold: 
 \begin{enumerate}[{\rm (a)}]
  \item For all $n\geq1$ we have $\Z_n\otimes_RJ\cong Z_n$ as (left) $J$-modules. 
  \item For all $n\geq1$ we have $\Hom_R(J,Z_n)\cong Z_n$ as $R$-modules.
 \end{enumerate}
\end{lemma}

\begin{proof} Note first that $Z_n\cong\Hom_R(Z_n,\Z_\infty)$ is a 
left $J=\End_R(\Z_\infty)$-module. 

 (a). We proceed by induction on $n$. For $n=1$ we apply the functor 
 $Z_1\otimes_R-$ to the short exact sequence $0\to UJ\to J\to Z_1\to 0$ coming from Lemma \ref{jjz}. Since $U$ is the annihilator of $Z_1$ we deduce $Z_1\otimes_RUJ=0$, so we get an isomorphism $Z_1\otimes_RJ\stackrel{\cong}\to\Z_1\otimes_RZ_1$, so $Z_1\otimes_RJ\cong Z_1$.
 
Now suppose $Z_n\otimes_RJ\cong Z_n$. Starting from the short exact sequence $0\to Z_n\to Z_{n+1}\to Z_1\to 0$ given by Lemma \ref{n+m/n}) we construct the commutative diagram with exact rows:
\[\diagram
&Z_n\otimes_RJ\rto\dto^{\cong}&Z_{n+1}\otimes_RJ\rto\dto&Z_1\otimes_RJ\rto\dto^{\cong}&0\\
0\rto&Z_n\rto&Z_{n+1}\rto&Z_1\rto&0
\enddiagram\]
whose vertical maps are obtained from the natural homomorphism \[-\otimes_RJ\cong\Hom_J(J,-)\otimes_RJ=\varphi^*\cdot\varphi_*\to\id_{\Rmd J},\] the last arrow coming from the adjunction.
Then the middle vertical arrow is an isomorphism too, proving the conclusion. 

(b). Using first the (proof of the) point (a), and second the adjunction isomorphism we obtain an isomorphism of $R$-modules
\begin{align*} \Hom_J(Z_n,Z_n)&\cong\Hom_J((\varphi^*\cdot\varphi_*)(Z_n),Z_n)\\ &\cong\Hom_R(\varphi_*(Z_n),\varphi_*(Z_n))=\Hom_R(Z_n,Z_n)\cong Z_n. 
\end{align*}

Combining it with the adjunction isomorphism between the functors 
\[\Hom_J(Z_n,-):\Rmd\rightleftarrows\Rmd R:Z_n\otimes_R-\] and the isomorphism of part (a) we get the isomorphisms of $R$-modules: 
\begin{align*}\Hom_R(J,Z_n)&\cong\Hom_R(J,\Hom_J(Z_n,Z_n))\\ &\cong\Hom_J(Z_n\otimes_RJ,Z_n)\cong\Hom_J(Z_n,Z_n)\cong Z_n
\end{align*} concluding the proof.
\end{proof}

\begin{theorem}\label{jss} With the notations above, if all conditions (1)-(5) are true,
then $\varphi:R\to J$ is a weak epimorphism of rings. Consequently 
$J$ is a self--small $R$-module.
\end{theorem}

\begin{proof} Using the isomorphism from the point (b) of Lemma
\ref{homjz} we get 
\[\Hom_R(J,J)=\Hom_R(J,\lim_{\leftarrow}Z_n)\cong\lim_{\leftarrow}\Hom_R(J,Z_n)\cong\lim_\leftarrow Z_n\cong J,\] therefore the ring homomorphism $\varphi:R\to J$ is a weak epimorphism. Then $J$ is self--small as $R$-module, by Proposition \ref{qepiss}. 
\end{proof}

\begin{example} Let $R=\Z$ and let $p$ be a prime. The direct system
\[\Z/p^1\to\Z/p^2\to\Z/p^3\to\ldots,\]
whose direct limit is the cocyclic abelian group $\Z/p^\infty$, satisfies the conditions (1)-(5). Thus Theorem \ref{jss} gives a proof that the group of
$p$-adic integers $\J_p=\Hom_\Z(\Z/p^\infty,\Z/p^\infty)$ is
self--small (for details, see also \cite{F}).
\end{example}

\begin{example} Let $R$ be a Dedekind ring, and let $\m$ be a maximal ideal. Put $Z_i=R/\m^i$, for all $i\geq1$. Then 
$S=Z_1$ is a simple $R$-module, and modules $Z_i$ are indecomposable, uniserial, with the composition series of the form 
\[0\subseteq Z_1\subseteq \ldots\subseteq Z_{i-1}\subseteq Z_i\] whose factors are all isomorphic to $S$. Moreover 
for every $i\geq1$ there is an exact sequence 
\[0\to S\to Z_{i+1}\to Z_i\to 0.\]
For more details concerning these modules we refer to \cite[1.4]{R04}. Then we obtain a direct system (DS) satisfying the 
conditions (1)-(5), so its inverse limit, the so called {\em $\m$-adic module}, $J=\displaystyle{\lim_\leftarrow}R/\m^i$ is self--small as $R$-module. 
\end{example}

{\sc Acknowledgements.} The author would like to thank Simion Breaz and Jan \v Zemli\v cka for a (quite old) discussion, in which they made him aware of the importance of constructing 
large self--small modules, as a counterpart of their work \cite{BS07}, and to the first named of them for some (more recent) valuable comments and suggestions.  Last but not least, the author is very indepted to Jan \v Stov\'\i\v cek for pointing him out a very serious  error made in a previous version of this paper.

\end{document}